\documentclass[a4paper,11pt]{article}

\RequirePackage{amsmath,amsthm,amssymb}
\usepackage{xcolor}
\usepackage{tikz}
\usetikzlibrary{patterns}
\usepackage{a4wide}
\usepackage{url}
\usepackage{hyperref}

\newtheorem{thm}{Theorem}
\newtheorem{lem}[thm]{Lemma}
\newtheorem{prop}[thm]{Proposition}
\newtheorem{corol}[thm]{Corollary}

\theoremstyle{definition}
\newtheorem{defn}[thm]{Definition}
\newtheorem{claim}[thm]{Claim}
\newtheorem{question}{Question}

\theoremstyle{remark}

\newcommand{\I}{\mathcal{I}}
\renewcommand{\P}{\mathcal{P}}
\newcommand{\T}{\mathcal{T}}
\renewcommand{\min}{\mathbf{min}}
\renewcommand{\max}{\mathbf{max}}

\newcommand{\mdd}{\mathbf{mdd}}
\newcommand{\Sat}{\mathbf{Sat}}
\newcommand{\dist}{\mathbf{dist}}

\newcommand{\red}{\mathbf{red}}
\newcommand{\uncolour}{\mathbf{uncolour}}
\newcommand{\tc}{\textcolor}
\renewcommand{\b}[1]{\tc{blue}{\underline{#1}}}
\renewcommand{\r}[1]{\tc{red}{\overline{#1}}}
\newcommand{\ISBT}{\mathbf{ISBT}}
\newcommand{\REC}{\mathbf{REC}}
\newcommand{\asc}{\mathbf{asc}}

\newcommand{\DrawPoint}[3][]{\filldraw[#1] (#2-0.5,#3+0.5) circle(6pt);}
\newcommand{\commentaire}[1]{}

\title{On minimal pattern-containing inversion sequences}

\begin{document}
\setlength{\tabcolsep}{2pt}
\renewcommand{\arraystretch}{1.15}

\begin{center}
{\LARGE \textbf{On minimal pattern-containing \\ inversion sequences}} \\[20pt]
{\large Benjamin Testart} \\[10pt]
{\small Université de Lorraine, CNRS, Inria, LORIA, F-54000 Nancy, France} \\
\end{center}

\abstract{
We introduce the notion of \emph{minimal inversion sequences} for a pattern $\rho$, which form the smallest set of inversion sequences whose avoidance is equivalent to the avoidance of $\rho$ for inversion sequences. We give a characterization of $\rho$-minimal inversion sequences based on the occurrences of the pattern $\rho$ they contain, and use it to find upper and lower bounds on the lengths of $\rho$-minimal inversion sequences. We provide some enumerative results on the exact number of minimal inversion sequences for some patterns, through a bijection with increasing trees, and some exhaustive generation. Lastly, we enumerate inversion sequences which are equal to their reduction, and find an interesting connection with poly-Bernoulli numbers.
} \medbreak


\section{Introduction}
\subsection{Basic definitions}
For all $n \geqslant 0$, let $\I_n$ be the set of \emph{inversion sequences} \cite{Corteel_Martinez_Savage_Weselcouch_2016, Mansour_Shattuck_2015} of length $n$, that is, the set of integer sequences $\sigma = (\sigma_1, \dots, \sigma_n)$ such that $\sigma_i \in \{0, \dots, i-1\}$ for all $i \in \{1, \dots, n\}$. Let $\I = \sqcup_{n \geqslant 0} \I_n$ be the set of all inversion sequences.

For all $n \geqslant 0$, let $\P_n$ be the set of \emph{Cayley permutations} \cite{Mor_Fraenkel_1984} (with values starting at 0) of length $n$, that is, the set of integer sequences $\rho = (\rho_1, \dots, \rho_n)$ whose set of values is exactly $\{0, \dots, \max(\rho)\}$. Let $\P = \sqcup_{n \geqslant 0} \P_n$ be the set of all Cayley permutations. In this work, a \emph{permutation} is a Cayley permutation which does not have any repeated values.

Given two integer sequences $\sigma$ of length $n$, and $\rho$ of length $k$ such that $1 \leq k \leq n$, we say that $\sigma$ \emph{contains} the \emph{pattern} $\rho$ if there exists a subsequence of $k$ entries of $\sigma$ whose values appear in the same order as the values of $\rho$. Such a subsequence is called an \emph{occurrence} of $\rho$. We say that $\sigma$ \emph{avoids} the pattern $\rho$ if $\sigma$ does not contain $\rho$. For instance, the sequence $497416$ contains two occurrences $494$ and $474$ of the pattern $010$, but avoids the pattern $012$. Given a set of integer sequences $S$, we denote by $\I(S)$ the set of inversion sequences which avoid every pattern in $S$.

\subsection{Motivation and outline}

The study of patterns in inversion sequences began in \cite{Mansour_Shattuck_2015} and \cite{Corteel_Martinez_Savage_Weselcouch_2016}, by analogy with permutation patterns. We begin with a short introduction to some classical concepts of the permutation patterns literature, which may be found in \cite{Bevan_2015}. The set of patterns which may appear in a permutation is (up to reduction\footnote{The reduction of a sequence is defined in Section \ref{sec_defs}.}) itself the set of permutations, yielding a nice correspondence between permutations and their patterns. More precisely, the pattern containment relation is a partial order on the set of all permutations. A \emph{permutation class} is a down-set in this partial order: if $\sigma$ and $\tau$ are two permutations such that $\sigma$ contains $\tau$ as a pattern, then any permutation class which contains $\sigma$ also contains $\tau$. Every permutation class can be characterized by the minimal set of permutations that it avoids; this set is called the \emph{basis} of the class. The basis of a permutation class is an antichain (a set of pairwise incomparable elements) under the pattern containment order, and may be infinite.

In the context of inversion sequences, we can observe patterns that are not inversion sequences (even after reduction). For instance, the inversion sequence $0021401$ contains the pattern 120. It is not difficult to see that the patterns which may appear in inversion sequences form (up to reduction) the set of nonempty Cayley permutations, and this set uniquely represents every way in which the values of an integer sequence may be ordered.
We may therefore use the term ``pattern" instead of ``nonempty Cayley permutation" from here on, as it is usual in the literature.

As we observed, Cayley permutations are not always inversion sequences, so one might wonder whether the set $\I(\rho)$ of inversion sequences avoiding a pattern $\rho$ can be described as the set of inversion sequences $\I(S_\rho)$ avoiding some set of \emph{inversion sequences} $S_\rho \subseteq \I$. If we allow infinite sets, the obvious answer is yes: a possible choice for $S_\rho$ is the set of all inversion sequences which contain the pattern $\rho$. A more interesting refinement of this question is to ask whether there is a \emph{finite} set (or equivalently, a finite antichain) $S_\rho$ satisfying the above, and how small can such a set be?

This question already surfaced in \cite[page 3]{Kotsireas_Mansour_Yildirim_2024}, although it is incorrectly answered. We recall their definition and approach, before introducing our own in the next section.
\begin{defn}[From \cite{Kotsireas_Mansour_Yildirim_2024}]
    Given a pattern $\tau = \tau_1 \dots \tau_k$, let $L_\tau$ be the set of all inversion sequences of the form $\theta^{(1)} \tau_1 \theta^{(2)} \tau_2 \dots \theta^{(k)} \tau_k$ such that the length of the inversion sequence $\theta^{(1)} \tau_1 \theta^{(2)} \tau_2 \dots \theta^{(j)} \tau_j$ is minimal for each $j \in \{1, \dots, k\}$. Note that some words $\theta^{(j)}$ might be empty.
\end{defn}
\begin{claim}[From \cite{Kotsireas_Mansour_Yildirim_2024}] \label{cl1}
    For any set of patterns $B$, an inversion sequence $\sigma$ avoids $B$ if and only if $\sigma$ avoids every pattern in $L_B :=\cup_{\{\tau \in B\}}L_\tau$.
\end{claim}

Our first contribution is the simple remark that Claim \ref{cl1} is false. Indeed, taking $B = \{10\}$, we have $L_B = \{010\}$, and Claim \ref{cl1} then states that an inversion sequence avoids the pattern 10 if and only if it avoids 010. Clearly this is not true: the sequence $0021 \in \I_4$ contains the pattern 10, yet avoids 010.

In \cite{Kotsireas_Mansour_Yildirim_2024}, Claim \ref{cl1} is then used to argue that inversion sequences containing a pattern $\tau \in \P_k$ for $k \geq 1$ must always contain a pattern $\tau' \in \I \cap \P$ of length at most $2k-1$ (since sequences in $L_\tau$ have length at most $2k-1$).
We give in Corollary \ref{corol_length} the correct upper bound of $3k - 2$, and we show in Section \ref{sec_optimality} that there are patterns $\tau$ of any length $k \geq 1$ for which this bound is reached.
This change of bound does not invalidate the algorithmic approach of \cite{Kotsireas_Mansour_Yildirim_2024}, since its purpose is to \emph{guess} succession rules, not to prove them; the enumerating formulas given in \cite{Kotsireas_Mansour_Yildirim_2024} are unaffected, as they do not rely on the validity of the $2k-1$ bound.

A summary of our work is as follows.
\begin{itemize}
    \item In Section \ref{sec_defs}, we introduce some definitions, most notably the notion of minimal inversion sequences for a pattern, and a useful statistic we call $\mdd$.
    \item In Section \ref{sec_structure}, we state properties on the general structure of minimal inversion sequences. In particular, we provide a characterization of $\rho$-minimal inversion sequences based on conditions satisfied by the occurrences of the pattern $\rho$ they contain. We also find lower and upper bounds on the lengths of $\rho$-minimal inversion sequences, which depend on the length and $\mdd$ of $\rho$.
    \item In Section \ref{sec_optimality}, we show that the aforementioned bounds are almost optimal.
    \item In Section \ref{sec_enumeration}, we enumerate $\rho$-minimal inversion sequences for any pattern $\rho$ such that $\rho_1 = \mdd(\rho)$ through a bijection with some increasing trees, and for any pattern $\rho$ of length at most 5 using computer programs.
    \item In Section \ref{sec_I_inter_P}, we enumerate the class $\I \cap \P$ (once again using a bijection with some increasing trees), and find new combinatorial interpretations of the poly-Bernoulli numbers of type C.
    \item In Section \ref{sec_conclusion}, we present open problems and  make some closing remarks.
\end{itemize}

\section{Definitions, notation, and simple observations} \label{sec_defs}
Given two integer sequences $\sigma$ and $\tau$, we denote $\tau \preceq \sigma$ if $\sigma$ contains $\tau$ as a pattern. Note that $\preceq$ is a partial order on the set $\P$, although it is only a preorder on the set of all integer sequences (it is also a preorder on the set of inversion sequences). In particular,
\begin{itemize}
    \item if $\tau \preceq \sigma$, then $|\tau| \leq |\sigma|$,
    \item if $\tau \preceq \sigma$ and $\sigma \npreceq \tau$, then $|\tau| < |\sigma|$.
\end{itemize}
We denote by $\sigma \cdot \tau$ the \emph{concatenation} of $\sigma$ and $\tau$, and by $\sigma^k$ the concatenation of $k$ copies of $\sigma$.
Given two integers $a$ and $b$, we denote by $[a,b]$ the \emph{integer interval} $\{k \in \mathbb Z \; : \; a \leq k \leq b\}$. We denote by $\varepsilon$ the empty sequence.

Let $\sigma$ be a nonnegative integer sequence of length $n \geq 1$. We define some functions on $\sigma$. Figure \ref{Fig_mdd_reduction} gives a visual representation of the functions $\mdd$ and $\red$ defined below.
\begin{itemize}
    \item Let $\dist(\sigma) := |\{\sigma_i\}_{i \in [1,n]}|$ be the number of \emph{distinct values} of $\sigma$.
    \item Let $\mdd(\sigma) := \max(\sigma_i - i + 1)_{i \in [1,n]}$ be the \emph{maximum diagonal difference} of $\sigma$, measuring how far $\sigma$ goes above the ``diagonal" sequence ${012\dots(n-1)}$. In particular, we have
    \begin{itemize}
        \item $\mdd(\sigma) \geq 0$ (since $\sigma_1 \geq 0$),
        \item $\mdd(\sigma) = 0$ if and only if $\sigma$ is an inversion sequence,
        \item $\mdd(\sigma \cdot \tau) = \max(\mdd(\sigma), \mdd(\tau) - n)$ for any nonempty sequence of nonnegative integers $\tau$.
    \end{itemize}
    \item Let $\red(\sigma)$ be the \emph{reduction} of $\sigma$, that is, the sequence obtained by replacing the minimum of $\sigma$ by 0, the second smallest value of $\sigma$ by 1, and so forth. Let ${\tau := \red(\sigma)}$. The following holds:
    \begin{itemize}
        \item $\tau_i \leq \sigma_i$ for all $i \in [1,n]$,
        \item $\mdd(\tau) \leq \mdd(\sigma)$,
        \item if $\sigma$ is an inversion sequence, then $\tau$ is an inversion sequence,
        \item $\dist(\tau) = \dist(\sigma)$,
        \item $\tau$ is the only Cayley permutation of length $n$ such that $\tau \preceq \sigma$,
        \item $\tau$ contains exactly the same patterns as $\sigma$.
    \end{itemize}
    Note that a pattern $\rho \in \P \backslash \{\varepsilon\}$ is contained in $\sigma$ if and only if $\rho$ is the reduction of a nonempty subsequence of $\sigma$.
    \item Let $\Sat(\sigma) := \{i \in [1,n] \; : \; \sigma_i - i + 1 = \mdd(\sigma)\}$ be the set of positions of \emph{saturated} entries of $\sigma$, that is, entries achieving the maximum diagonal difference. In particular,
    \begin{itemize}
        \item every saturated entry of $\sigma$ is a strict left-to-right maximum (i.e. it has a strictly greater value than every entry to its left),
        \item the maximum of $\sigma$ is not always saturated (e.g. $\sigma = 001$),
        \item if $\sigma \in \I_n$, $i \in [1,n]$ and $\sigma' := (\sigma_j)_{j \in [1,n] \backslash \{i\}}$ is the sequence obtained by deleting the entry $\sigma_i$ from $\sigma$, then $\sigma'$ is an inversion sequence if and only if $\Sat(\sigma) \subseteq [1,i]$ (since deleting $\sigma_i$ increases the diagonal difference of every entry to its right by 1).
    \end{itemize}
\end{itemize}

\begin{figure}[ht]
\centering
\begin{tikzpicture}
\begin{scope}[scale = 0.6]
    \draw[line width = 2pt, gray] (0,0) -- (5,5);
    \draw (0,0) grid (5,7);
    \draw[line width = 2pt, dotted, gray] (0.5,0.5) -- (0.5,2.5);
    \draw[line width = 2pt, dotted, gray] (1.5,1.5) -- (1.5,0.5);
    \draw[line width = 2pt, dotted, gray] (3.5,3.5) -- (3.5,6.5);
    \draw[line width = 2pt, dotted, gray] (4.5,4.5) -- (4.5,0.5);
    \DrawPoint{1}{2};
    \DrawPoint{2}{0};
    \DrawPoint{3}{2};
    \DrawPoint{4}{6};
    \DrawPoint{5}{0};
    \draw (2.5,-0.6) node[below, font=\small] {Diagonal difference};
    \draw (0.5,0) node[below] {2};
    \draw (1.5,0) node[below] {-1};
    \draw (2.5,0) node[below] {0};
    \draw (3.5,0) node[below] {3};
    \draw (4.5,0) node[below] {-4};

    \draw [->, line width = 1.5pt] (5.3, 2.5) -- (8.7, 2.5);
    \draw (7,2.5) node[above, font=\small] {Reduction};

\begin{scope}[xshift = 9cm]
    \draw[line width = 2pt, gray] (0,0) -- (5,5);
    \draw (0,0) grid (5,5);
    \draw[line width = 2pt, dotted, gray] (0.5,0.5) -- (0.5,1.5);
    \draw[line width = 2pt, dotted, gray] (1.5,1.5) -- (1.5,0.5);
    \draw[line width = 2pt, dotted, gray] (2.5,2.5) -- (2.5,1.5);
    \draw[line width = 2pt, dotted, gray] (3.5,3.5) -- (3.5,2.5);
    \draw[line width = 2pt, dotted, gray] (4.5,4.5) -- (4.5,0.5);
    \DrawPoint{1}{1};
    \DrawPoint{2}{0};
    \DrawPoint{3}{1};
    \DrawPoint{4}{2};
    \DrawPoint{5}{0};
    \draw (2.5,-0.6) node[below, font=\small] {Diagonal difference};
    \draw (0.5,0) node[below] {1};
    \draw (1.5,0) node[below] {-1};
    \draw (2.5,0) node[below] {-1};
    \draw (3.5,0) node[below] {-1};
    \draw (4.5,0) node[below] {-4};
\end{scope}
\end{scope}
\end{tikzpicture}
\caption{On the left, the sequence 20260 of maximum diagonal difference 3. On the right, its reduction 10120 of maximum diagonal difference 1.}
\label{Fig_mdd_reduction}
\end{figure}

\begin{defn}
Let $\rho \in \P \backslash \{\varepsilon\}$ be a pattern, and let $\I\P[\rho] := \{\sigma \in \I \cap \P \; : \; \rho \preceq \sigma\}$ be the set of inversion sequences which are also Cayley permutations and contain the pattern $\rho$. We say that $\sigma$ is a \emph{$\rho$-minimal inversion sequence} if $\sigma$ is a minimal element of the partially ordered set $(\I\P[\rho], \preceq)$.
\end{defn}

Let $\sigma \in \I$ be an inversion sequence, and let $\rho \in \P \backslash \{\varepsilon\}$ be a pattern. Then
\begin{itemize}
    \item $\rho \preceq \sigma$ if and only if there exists a $\rho$-minimal inversion sequence $\tau$ such that $\tau \preceq \sigma$, since $\red(\sigma) \preceq \sigma$, and $\rho \preceq \sigma$ implies that $\red(\sigma) \in \I\P[\rho]$,
    \item $\sigma$ is $\rho$-minimal if and only if $\rho \preceq \sigma$, $\sigma \in \P$, and there is no inversion sequence $\tau$ such that $|\tau| < |\sigma|$ and $\rho \preceq \tau \preceq \sigma$.
\end{itemize}
By the first bullet point above, $\sigma$ avoids $\rho$ if and only if $\sigma$ avoids every $\rho$-minimal inversion sequence. It is worth noting that $\rho$-minimal inversion sequences form the smallest set of inversion sequences satisfying this property for any $\sigma \in \I$. More precisely, any set $S \subseteq \I \cap \P \backslash \{\varepsilon\}$ such that $\I(S) = \I(\rho)$ must contain every $\rho$-minimal inversion sequence.

A given pattern $\rho$ may admit several $\rho$-minimal inversion sequences, and their lengths can be different. For instance, the $10$-minimal inversion sequences are $\{010, 0021\}$, and the $0201$-minimal inversion sequences are $\{00201, 01201, 001312, 010312, 011302\}$, as we will see. We show in Corollary \ref{corol_length} that there is only a finite number of $\rho$-minimal inversion sequences for any pattern $\rho$, by giving an upper bound on their length. In general, it is no easy task to find every $\rho$-minimal inversion sequence, although we provide such results for some patterns in Section \ref{sec_enumeration}.

\section{Structure of minimal inversion sequences} \label{sec_structure}

Given a pattern $\rho$, a $\rho$-minimal inversion sequence $\sigma$, and an occurrence of $\rho$ in $\sigma$, it can be seen that every entry of $\sigma$ serves one of two purposes: either it is part of the occurrence of $\rho$, or it ``fills up" space in order to allow the placement of larger values to its right (since the value of every entry of an inversion sequence must be less than its position). For instance, in the 10-minimal inversion sequence 0021, the subsequence 21 forms an occurrence of the pattern 10, while the entries 00 are used as ``filler".
For a non-trivial example, consider the $1306524$-minimal inversion sequence $01214027635$. This sequence contains three occurrences of the pattern $1306524$. For each occurrence, we colour the remaining filler entries in gray: $\tc{gray}01\tc{gray}2\tc{gray}140\tc{gray}27635$, $\tc{gray}0\tc{gray}12\tc{gray}140\tc{gray}27635$, $\tc{gray}0\tc{gray}1\tc{gray}2140\tc{gray}27635$.

Proposition \ref{prop_minimal_characterization} below shows that the $\rho$-minimality of an inversion sequence can be expressed by conditions on its filler entries for each occurrence of $\rho$.

\begin{prop} \label{prop_minimal_characterization}
    Let $n,k \geq 1$, let $\rho \in \P_k$, and let $\sigma \in \I_n \cap \P_n$ be a sequence containing the pattern $\rho$. Then $\sigma$ is $\rho$-minimal if and only if for any set $R \subseteq [1,n]$ such that $\rho' := (\sigma_{i})_{i \in R}$ is an occurrence of $\rho$, both conditions below are satisfied:
    \begin{enumerate}
        \item $[\max(\Sat(\sigma)),n] \subseteq R$,
        \item any value of $\sigma$ which does not appear in $\rho'$ appears at least twice in $\sigma$.
    \end{enumerate}
\end{prop}
\begin{proof}
    Let $s := \max(\Sat(\sigma))$. First, we prove by contrapositive that both conditions are necessary. Let $R \subseteq [1,n]$ be such that $\rho' := (\sigma_{i})_{i \in R}$ is an occurrence of $\rho$.
    \begin{enumerate} 
        \item Assume that $u \in [s,n]$ satisfies $u \notin R$. Then $(\sigma_i)_{i \in [1,n] \backslash \{u\}}$ is an inversion sequence which contains the pattern $\rho$, hence $\sigma$ is not $\rho$-minimal.
        \item Assume that $u \in [1,n] \backslash R$ is such that the value $\sigma_u$ appears only once in $\sigma$, and let $ \sigma' := \red((\sigma_i)_{i \in [1,n] \backslash \{u\}})$. Clearly, $\sigma'$ contains the pattern $\rho$. We next show that $\sigma'$ is an inversion sequence. For each $i \in [1,n] \backslash \{u\}$, the transformation from $\sigma$ to $\sigma'$ shifts the diagonal difference of each entry $\sigma_i$ as follows:
        \begin{itemize}
            \item if $i > u$, then the removal of $\sigma_u$ increases the diagonal difference of $\sigma_i$ by 1,
            \item if $\sigma_i > \sigma_u$, then the reduction decreases the diagonal difference of $\sigma_i$ by 1.
        \end{itemize}
        It follows that the only entries whose diagonal difference is actually increased are those $\sigma_i$ such that $i > u$ and $\sigma_i <  \sigma_u$. Since $\sigma_u < u$, such entries satisfy $\sigma_i < i-2$, so they remain below the diagonal in $\sigma'$. We conclude that $\sigma'$ is an inversion sequence, therefore $\sigma$ is not $\rho$-minimal.
    \end{enumerate}
    Now assume that both conditions are satisfied by every occurrence of $\rho$ in $\sigma$. To complete our proof, we show that there is no sequence $\sigma' \in \I \cap \P$ such that $\sigma' \neq \sigma$ and ${\rho \preceq \sigma' \preceq \sigma}$.
    Let $J \varsubsetneq [1,n]$ be such that $\sigma' := \red((\sigma_i)_{i \in J})$ contains the pattern $\rho$, and let ${J^c := [1,n] \backslash J}$. By condition 1, we have $J^c \subseteq [1,s-1]$, therefore $\sigma_i \leq s-2$ for all $i \in J^c$. Let $V := \{v \in \{\sigma_i\}_{i \in J^c} \; : \; v \notin \{\sigma_i\}_{i \in J} \}$ be the set of values of $\sigma$ which do not appear in $\sigma'$. By condition 2, we have $|J^c| \geq 2 \,|V|$. Since $J^c \neq \emptyset$, this implies that $|J^c| > |V|$.
    Recall that $\sigma_s = s-1$, and for all $i \in J^c$ we have both $i < s$ and $\sigma_i < \sigma_s$. By looking at the value to which $\sigma_s$ is mapped in $\sigma'$, we obtain
    $$\sigma'_{s-|J^c|} = s-1 - |V| > s-1-|J^c|,$$
    therefore $\sigma'$ is not an inversion sequence. We conclude that $\sigma$ is $\rho$-minimal.
\end{proof}

Next, we establish a correspondence between the saturated entries of a pattern $\rho$ and those of $\rho$-minimal inversion sequences. This will be useful to prove Lemma \ref{lem_len-dist-mdd}, which shows a relation between three statistics (namely the length, $\dist$, and $\mdd$) of a pattern $\rho$ and its minimal inversion sequences. This relation then yields some lower and upper bounds on the lengths of $\rho$-minimal inversion sequences, in Corollary \ref{corol_length}.

\begin{prop} \label{prop_sat_pattern}
    Let $n,k \geq 1$, let $\rho \in \P_k$, let $\sigma \in \I_n$, let ${\ell \in [0,k-1]}$ be such that $k-\ell \in \Sat(\rho)$, let ${r_1 < \dots < r_k \in [1,n]}$ be such that $\rho' := (\sigma_{r_i})_{i \in [1,k]}$ is an occurrence of $\rho$, and let $R := \{r_i\}_{i \in [1,k]}$. If $\sigma$ is $\rho$-minimal, the following holds:
    \begin{enumerate}
        \item for all $i \in [1,n]$, if $\sigma_i \geq \sigma_{r_{k-\ell}}$, then $i \in R$,
        \item for all $i \in [0,\ell]$, we have $r_{k-i} = n-i$,
        \item $n - \ell \in \Sat(\sigma)$.
    \end{enumerate}
\end{prop}
\begin{proof}
    First, we prove items 1 and 2 by contrapositive. Note that item 2 is equivalent to $[r_{k-\ell}, n] \subseteq R$. Assume that $\sigma$ does not satisfy item 1 or item 2 (or both), i.e.
    $$U := \{i \in [1,n] \backslash R \; : \; \sigma_i \geq \sigma_{r_{k-\ell}} \; \text{ or } \; i \geq r_{k-\ell}\} \neq \emptyset.$$
    In particular, we have $\min(U) > \sigma_{r_{k-\ell}}$ since every $i \in U$ satisfies either ${i > \sigma_i \geq \sigma_{r_{k-\ell}}}$ or $i \geq r_{k-\ell} > \sigma_{r_{k-\ell}}$ because $\sigma$ is an inversion sequence.
    Let $\sigma' := \red(\{\sigma_i\}_{i \in [1,n] \backslash U})$, let $n' := |\sigma'|$, let ${r'_i := r_i - |\{j \in U \; : \; j < r_i\}|}$ for all ${i \in [1,k]}$, and let $R' := \{r'_i\}_{i \in [1,k]}$.
    By construction, $r'_{k-i} = n'-i$ for all $i \in [0,\ell]$, and $\rho'' := (\sigma'_i)_{i \in R'}$ is an occurrence of $\rho$. Note that every entry $\sigma'_i$ corresponds to some entry $\sigma_j$ such that $i \leq j$ and $\sigma'_i \leq \sigma_j$; to be precise, $j$ can be defined as the only element of $[1,n] \backslash U$ such that $i = j - |\{u \in U \; : \; u < j\}|$. Now we show that $\sigma'$ is an inversion sequence. 
    \begin{itemize}
        \item Let $i \in [1,n']$ be such that $\sigma'_i \leq \sigma_{r_{k-\ell}}$, and let $\sigma_j$ be the entry of $\sigma$ corresponding to $\sigma'_i$. Since $\min(U) > \sigma_{r_{k-\ell}}$, the entry $\sigma_j$ either stays at position $j$, or moves to a position at least $\sigma_{r_{k-\ell}}+1$ in $\sigma'$. Hence we have either $i = j > \sigma_j \geq \sigma'_i$, or $i \geq \sigma_{r_{k-\ell}}+1 > \sigma'_i$. In any case, $\sigma'_i < i$.
        \item  Let $i \in [0,n'-1]$ be such that $\sigma'_{n'-i} > \sigma_{r_{k-\ell}}$. By definition of $\sigma'$ and $R'$, we have ${n'-i \in R'}$. Because $k-\ell \in \Sat(\rho)$ and $\sigma'_{n'-i} > \sigma_{r_{k-\ell}} \geq \sigma'_{r'_{k-\ell}}$, we have ${n'-i > r'_{k-\ell} = n'-\ell}$, hence $n'-i = r'_{k-i}$.
        By definition of $U$, every entry of $\sigma'$ of value at least $\sigma'_{n'-\ell}$ is part of $\rho''$. Since $\sigma'$ is a Cayley permutation, this implies that ${\sigma'_{n'-i} - \sigma'_{n'-\ell} = \rho_{k-i} - \rho_{k-\ell}}$. Moreover, we have $\rho_{k - i} - \rho_{k-\ell} \leq \ell - i$ because ${k-\ell \in \Sat(\rho)}$.
        Since $\sigma'_{n'-\ell} = \sigma'_{r'_{k-\ell}} \leq \sigma_{r_{k-\ell}}$, we have $\sigma'_{n'-\ell} < n'-\ell$ by the previous bullet point. We conclude that $\sigma'_{n'-i} \leq \sigma'_{n'-\ell} + \ell - i < n'-i$.
    \end{itemize}
    We showed that $\sigma'_i < i$ for every $i \in [1,n']$, hence $\sigma'$ is an inversion sequence, and therefore $\sigma$ is not $\rho$-minimal. This concludes the proof of items 1 and 2.

    Now we prove item 3 by contrapositive.
    Assume that $\sigma$ satisfies items 1 and 2, that $\sigma \in \P$ (these conditions are safe to assume since they are necessary to be $\rho$-minimal), and that $\sigma$ does not satisfy item 3. Since $\sigma$ is an inversion sequence and $n-\ell \notin \Sat(\sigma)$, we must have $\sigma_{n-\ell} < n - \ell - 1$. By items 1 and 2, every entry of $\sigma$ of value at least $n-\ell$ is part of $\rho'$. Since $\sigma$ is a Cayley permutation, by the same reasoning used in the first part of the proof, for all $i \in [0, \ell-1]$ such that $\sigma_{n-i} > n-\ell$, we have $$\sigma_{n-i} = \sigma_{n-\ell} + \rho_{k-i} - \rho_{k-\ell} \leq \sigma_{n-\ell} +  \ell - i < n-i-1,$$
    hence the entries of $\sigma$ at positions in $[n-\ell, n]$ are not saturated.
    Now, assume that $i \in [1, k-\ell-1]$ is such that $r_i \in \Sat(\sigma)$.
    In particular, we have $\sigma_{r_i} - r_i > \sigma_{r_{k-\ell}} - r_{k-\ell}$ since $r_{k-\ell} = n - \ell$ by item 2. Because $k-\ell \in \Sat(\rho)$, we have $\rho_i - i \leq \rho_{k- \ell} - (k-\ell)$, and in particular $\rho_i < \rho_{k- \ell}$. Because $(\sigma_{r_i})_{i \in [1,k]}$ is an occurrence of $\rho$, we have $\sigma_{r_i} < \sigma_{r_{k-\ell}}$ and $\sigma_{r_{k-\ell}} - \sigma_{r_i} \geq \rho_{k-\ell} - \rho_i \geq k - \ell - i$. It follows that $r_{k-\ell} - r_i > k - \ell - i$, in other words there exists some $j \in [r_i, r_{k-\ell}]$ such that $j \notin R$. In conclusion, any saturated entry of $\sigma$ that is part of $\rho'$ must be to the left of some entry of $\sigma$ that is not part of $\rho'$. This contradicts the first item of Proposition \ref{prop_minimal_characterization}, hence $\sigma$ is not $\rho$-minimal.
\end{proof}

\begin{lem} \label{lem_len-dist-mdd}
    Let $\rho$ be a pattern, and let $\sigma$ be a $\rho$-minimal inversion sequence. Then $$|\sigma| - \dist(\sigma) = |\rho| - \dist(\rho) + \mdd(\rho).$$
\end{lem}
\begin{proof}
    Let $n := |\sigma|$, let $k := |\rho|$, let $\ell \in [0,k-1]$ be such that $k-\ell \in \Sat(\rho)$, let $R \subseteq [1,n]$ be such that $\rho' := (\sigma_{i})_{i \in R}$ is an occurrence of $\rho$, and let $\tau := (\sigma_i)_{i \in [1,n] \backslash R}$ be the sequence obtained by removing $\rho'$ from $\sigma$.

    By Proposition \ref{prop_sat_pattern}, we know that $\sigma_{n-\ell}$ is $\rho'_{k-\ell}$, and every entry of $\tau$ has value less than $\sigma_{n-\ell}$. Since $\sigma$ is a Cayley permutation, this implies that ${\sigma_{n-\ell} = \rho_{k-\ell} + \dist(\sigma) - \dist(\rho)}$, where the number $\dist(\sigma) - \dist(\rho)$ counts the values of $\tau$ which do not appear in $\rho'$. We conclude by noticing that
    \begin{itemize}
        \item $\sigma_{n-\ell} = n - \ell - 1$, by item 3 of Proposition \ref{prop_sat_pattern}, and
        \item $\rho_{k-\ell} = \mdd(\rho) + k-\ell - 1$, by definition of $\ell$,
    \end{itemize}
    hence $n = \mdd(\rho) + k + \dist(\sigma) - \dist(\rho)$.
\end{proof}

\begin{corol} \label{corol_length}
    Let $\rho$ be a pattern, and let $\sigma$ be a $\rho$-minimal inversion sequence. Then $|\rho| + \mdd(\rho) \leq |\sigma| \leq |\rho| + 2 \, \mdd(\rho)$. In particular, $|\sigma| \leq 3\,|\rho| - 2$.
\end{corol}
\begin{proof}
    The lower bound is satisfied because $|\dist(\sigma)| \geq |\dist(\rho)|$, since $\sigma$ contains the pattern $\rho$.

    To prove the upper bound, let $\rho'$ be an occurrence of $\rho$ in $\sigma$, and let $\tau$ be the sequence obtained by removing $\rho'$ from $\sigma$. Let $t := |\tau|$, $m := \mdd(\rho)$, and let $v$ be the number of values of $\tau$ which do not appear in $\rho'$. In particular, $t = |\sigma| - |\rho|$, and $v = \dist(\sigma) - \dist(\rho)$. By item 2 of Proposition \ref{prop_minimal_characterization}, any value of $\tau$ which is not in $\rho'$ must appear at least twice in $\tau$, therefore $t \geq 2v$. By Lemma \ref{lem_len-dist-mdd}, we have $t = v + m$, hence $2v \leq v + m$, hence $v \leq m$, hence $t \leq 2m$, and finally $|\sigma| = |\rho| + t \leq |\rho| + 2m$.

    To see that $|\sigma| \leq 3\,|\rho| - 2$, simply notice that $\mdd(\rho) \leq \max(\rho)$, and $\max(\rho) \leq |\rho|-1$ since $\rho$ is a Cayley permutation.
\end{proof}

For any pattern $\rho$, we remark that all inversion sequences of the form $\alpha \cdot \rho$ for some ${\alpha \in \I_{\mdd(\rho)}}$ are $\rho$-minimal. Indeed, these are $\rho$-containing inversion sequences of minimal length, and they are Cayley permutations since $\max(\alpha) < \mdd(\rho) \leq \max(\rho)$. In particular, the total number of $\rho$-minimal inversion sequences is at least $\mdd(\rho)!$.

We conclude this section with a very simple but important consequence of Corollary \ref{corol_length} on pattern avoidance in inversion sequences.

\begin{thm}
    Let $P \subset \P \backslash \{\varepsilon\}$ be a finite set of patterns. There exists a finite set $Q \subset \I \cap \P$ such that $\I(P) = \I(Q)$.
\end{thm}
\begin{proof}
    Let $Q$ be the set obtained by replacing each pattern $\rho \in P$ by the set of $\rho$-minimal inversion sequences. By definition of $\rho$-minimal inversion sequences, we have $Q \subset \I \cap \P$ and $\I(Q) = \I(P)$. By Corollary \ref{corol_length}, $Q$ contains only inversion sequences of length at most $\max(3 \, |\rho| - 2 \; : \; \rho \in P)$, therefore $Q$ is finite.
\end{proof}

\section{Optimality} \label{sec_optimality}
In this section, we examine the optimality of the bounds on the lengths of $\rho$-minimal inversion sequences obtained in Corollary \ref{corol_length}. First, we remark that the upper bound depending only on $|\rho|$ cannot be improved, since it can be reached for a pattern $\rho$ of any length $k \geq 1$. Indeed, if $\rho$ is a pattern of length $k$ starting with the value $k-1$, take for example $\rho = (k-1)(k-2) \dots 210$, then
$$001122\dots (k-3)(k-3)(k-2)(k-2) \cdot (\rho_i + k-1)_{i \in [1,k]}$$
is a $\rho$-minimal inversion sequence of length $3k-2$. See Proposition \ref{prop_tight_ub} below for a proof.

Next, we investigate the lower and upper bounds depending on both $|\rho|$ and $\mdd(\rho)$. Clearly, the lower bound $|\rho| + \mdd(\rho)$ is reached for any pattern $\rho$, since $0^{\mdd(\rho)} \cdot \rho$ is always an inversion sequence. Now we show that the upper bound $|\rho| + 2\,\mdd(\rho)$ is reached for most patterns.

\begin{prop} \label{prop_tight_ub}
    Let $k \geq 2$, let $\rho \in \P_k$, and let $m := \mdd(\rho)$. If $\rho_i = 0$ for some $i \geq 2$, then there exists a $\rho$-minimal inversion sequence of length $k + 2m$.
\end{prop}
\begin{proof}
    First assume that $\rho_1 \neq 0$ or $\rho_z = 0$ for some $z \geq 3$, and let $$\sigma := 0011 \dots (m-1)(m-1) \cdot (\rho_i + m)_{i \in [1,k]} \in \I_{2m + k} \cap \P.$$
    In that case, $(\sigma_i)_{i \in [2m+1, 2m+k]}$ is the only occurrence of $\rho$ in $\sigma$, and it satisfies both conditions of Proposition \ref{prop_minimal_characterization}, therefore $\sigma$ is $\rho$-minimal.

     Now assume that $\rho_1 = \rho_2 = 0$ and $\rho_i > 0$ for all $i \in [3,k]$. If $k = 2$, then $00$ is trivially the only $\rho$-minimal inversion sequence. Otherwise, let $$\tau := 01 \dots (m-1)mm(m-1) \dots 10 \cdot (\rho_i + m)_{i \in [3,k]} \in \I_{2m + k} \cap \P.$$
     Notice that $\red(\rho_1 \rho_2 \rho_3) = 001$, and the prefix $(\tau_i)_{i \in [1, 2m+2]}$ avoids the pattern 001, hence the last $k-2$ entries of any occurrence of $\rho$ in $\tau$ must be $(\tau_i)_{i \in [2m+3, 2m+k]}$. 
     It follows that $\tau$ contains exactly $m+1$ occurrences of $\rho$, and they are the sequences $zz \cdot (\rho_i + m)_{i \in [3,k]}$ for $z \in [0,m]$. All of these occurrences satisfy both conditions of Proposition \ref{prop_minimal_characterization}, hence $\tau$ is $\rho$-minimal.
\end{proof}

We observe that there exist patterns $\rho$ whose minimal inversion sequences are all of length strictly less than $|\rho| + 2\, \mdd(\rho)$. The smallest such pattern is $021$, and its minimal inversion sequences are $\{0021, 0121\}$.
Now we show that even when the upper bound $|\rho| + 2\,\mdd(\rho)$ is not reached, there are still some $\rho$-minimal inversion sequences whose length comes very close to this bound.

\begin{prop} \label{prop_min_max_length}
    For any pattern $\rho$, there exists a $\rho$-minimal inversion sequence of length at least $|\rho| + 2 \, \mdd(\rho) - 2$.
\end{prop}
\begin{proof}
    Let $k := |\rho|$, and let $m := \mdd(\rho)$. If $m \leq 2$, then $0^m \cdot \rho$ is a $\rho$-minimal inversion sequence of length $k + m \geq k + 2m - 2$. Assume now that $m \geq 3$. If $\rho_i = 0$ for some $i \geq 2$, then there exists a $\rho$-minimal inversion sequence of length $k + 2m$ by Proposition \ref{prop_tight_ub}. Assume now that $\rho_1$ is the only zero of $\rho$.

    Assume that $\rho_2 \geq 2$. Let $\rho' := (\rho_i + m - 1)_{i \in[2,k]}$, which is a sequence such that $\mdd(\rho') = 2m$, and let $\tau := 001122 \dots (m-1) (m-1) \cdot \rho' \in \I_{k+2m-1}$. Every term of value less than $m$ in $\tau$ can only take the role of the 0 in an occurrence of $\rho$, so all occurrences of $\rho$ in $\tau$ are subsequences of the form $v \cdot \rho'$ for some $v \in [0,m-1]$. These occurrences all satisfy the conditions of Proposition \ref{prop_minimal_characterization}, hence $\tau$ is $\rho$-minimal.

    Assume that $\rho_2 = 1$, and let $\ell := \max(i \in [1,k] \; : \; (\rho_j)_{j \in [1,i]} \in \I)$ be the length of the longest prefix of $\rho$ that is an inversion sequence. In particular, $\ell \geq 2$, $\rho_i < \ell$ for all $i \in [1,\ell]$, $\rho_{\ell+1} > \ell$, and there exists some $q \in [\ell+2,k]$ such that $\rho_q = \ell$.
    Let $\alpha := (\rho_i)_{i \in [1,\ell-1]}$, let $\beta := (\ell+1)(\ell+1)(\ell+2)(\ell+2) \dots (\ell+m-2)(\ell+m-2)$, and let $\gamma$ be the sequence of length $k-\ell+1$ defined by
    $$\forall i \in [\ell,k], \; \gamma_{i-\ell+1} := \begin{cases}
        \rho_{i} & \text{ if} \quad \rho_{i} \leq \ell, \\
        \rho_{i} + m-2 & \text{ if} \quad \rho_{i} > \ell.
    \end{cases}$$
    Let $\sigma := 00 \cdot \alpha \cdot \beta \cdot \gamma$, 
    which is a sequence of length $n := k+2m-2$. For example, if $\rho = 01165342$, then we have $m = 3$, $\ell = 3$, $\alpha = 01$, $\beta = 44$, and $\gamma = 176352$, hence $\sigma = 000144176352$.
    Going back to the general case, note that $\mdd(\alpha) = 0$, $\mdd(\beta) = \ell+1 = |00 \cdot\alpha|$, and $$\mdd(\gamma) = m+(\ell-1)+(m-2) = 2m+\ell-3 = |00 \cdot \alpha \cdot \beta|,$$ therefore $\sigma$ is an inversion sequence.
    It can be seen that $\sigma$ is also a Cayley permutation:
    \begin{itemize}
        \item every value in $[0, \ell]$ appears in $\rho$, and remains unchanged in $\alpha \cdot \gamma$,
        \item every value in $[\ell+1, \ell+m-2]$ appears in $\beta$,
        \item every value in $[\ell+m-1, \max(\rho)+m-2]$ appears in $\gamma$,
        \item by construction, $\max(\sigma) = \max(\gamma) = \max(\rho)+m-2$.
    \end{itemize}
    Now, we show that in an occurrence of $\rho$ in $\sigma$, the values of $\beta$ are too large to take the role of the term $\rho_i$ for any $i \leq \ell$. Let $r_1 < \dots < r_k \in [1,n]$ be such that $(\sigma_{r_i})_{i \in [1,k]}$ is an occurrence of $\rho$. Notice that $r_2 \geq 4$, since $\sigma_1 = \sigma_2 = \sigma_3 = 0$ and $\rho_1$ is the only zero of $\rho$. This implies that $r_i \geq i+2$ for all $i \in [2,k]$, in particular $r_{\ell} \geq \ell+2$. Let $b := |00 \cdot \alpha \cdot \beta| = 2m+\ell-3$, and let $q \in [\ell+2,k]$ be such that $\rho_q = \ell$.
    \begin{itemize}
        \item Assume for the sake of contradiction that $r_\ell \leq b$ and $r_{\ell+1} > b$. In particular, $\sigma_{r_{\ell}} > \ell$ since $r_\ell \in [\ell+2, b]$. Because $\rho_q > \rho_\ell$, this also implies that $\sigma_{r_q} > \ell$. Now, we have
        \begin{itemize}
            \item $\{i \; : \; \rho_i \geq \ell \} = \{i \; : \; \sigma_{r_i} \geq \sigma_{r_q}\}$ by definition of $(r_i)_{i \in [1,k]}$ and because $\rho_q = \ell$,
            \item $|\{i \; : \; \sigma_{r_i} \geq \sigma_{r_q}\}| \leq |\{i \; : \; \gamma_i \geq \sigma_{r_q}\}|$ since 
            \begin{itemize}
                \item $\sigma_{r_i} \geq \sigma_{r_q}$ is equivalent to $\rho_i \geq \ell$, so it implies that $i > \ell$ by definition of $\ell$,
                \item $r_{\ell+1} > b$ by hypothesis, so $(\sigma_{r_i})_{i \in [\ell+1,k]}$ is a subsequence of $\gamma$,
            \end{itemize}
            \item $|\{i \; : \; \gamma_i \geq \sigma_{r_q}\}| \leq |\{i \; : \; \gamma_i > \ell\}|$ since $\ell < \sigma_{r_q}$,
            \item $|\{i \; : \; \gamma_i > \ell\}| = |\{i \; : \; \rho_i > \ell\}|$ by definition of $\gamma$.
        \end{itemize}
        In summary, we have $|\{i \; : \; \rho_i \geq \ell \}| \leq |\{i \; : \; \rho_i > \ell\}|$; this implies that $\rho$ does not contain the value $\ell$, a contradiction.
        \item Assume for the sake of contradiction that $r_{\ell+1} \leq b$. Recall that $\sigma_{r_\ell} < \sigma_{r_q} < \sigma_{r_{\ell+1}}$, and $\ell + 2 \leq r_\ell < r_{\ell+1} < r_q$. Since $\ell+2 \leq r_\ell < r_{\ell+1} \leq b$ and $\beta$ is nondecreasing, we have $\ell+1 \leq \sigma_{r_\ell} \leq \sigma_{r_{\ell+1}} \leq \ell + m - 2$. By construction of $\sigma$, every term $\sigma_i$ for $i \in [1+r_{\ell+1}, n]$ falls in one of three cases:
        \begin{itemize}
            \item $\sigma_i \geq \sigma_{r_{\ell+1}}$ if $i \leq b$ since $\beta$ is nondecreasing,
            \item $\sigma_i \leq \ell < \sigma_{r_\ell}$ if $i > b$ and $\rho_{i+k-n} \leq \ell$,
            \item $\sigma_i > \ell+m-2 \geq \sigma_{r_{\ell+1}}$ if $i > b$ and $\rho_{i+k-n} > \ell$.
        \end{itemize}
        This contradicts the existence of $\sigma_{r_q}$.
    \end{itemize}
    We showed that $r_\ell > b$. Since $|[b+1,n]| = |[\ell,k]|$, we have ${r_i = i+b-\ell+1 = i+n-k}$ for all $i \in [\ell,k]$. In particular, $\sigma_{r_q} = \sigma_{q+b-\ell+1} = \gamma_{q-\ell+1} = \rho_q = \ell$, hence ${\sigma_{r_{\ell-1}} < \ell}$ since ${\rho_{\ell-1} < \rho_q}$. This implies that $r_{\ell-1} \leq \ell+1$, since $r_{\ell-1} < r_\ell = b+1$ and $\sigma_i > \ell$ for all ${i \in [\ell+2,b]}$. Lastly, recalling that $r_i \geq i+2$ for all $i \geq 2$, we have ${r_j = j+2}$ for all $j \in [2,\ell-1]$. In summary, there are at most three occurrences of $\rho$ in $\sigma$, as we completely determined the value of each $r_i$ except for $r_1 \in \{1,2,3\}$. It is easy to verify that those three subsequences are occurrences of $\rho$: indeed, all three are equal to the sequence $\alpha \cdot \gamma$, which is order-isomorphic to $\rho$ by definition. It is also straightforward to verify that all three occurrences of $\rho$ satisfy the conditions of Proposition \ref{prop_minimal_characterization}, since ${\max(\Sat(\sigma)) = b+\max(\Sat(\rho)) \geq b+2}$, and the entries of $\sigma$ which are not used in each occurrence of $\rho$ are exactly two zeros and the factor $\beta$. We conclude that $\sigma$ is $\rho$-minimal.
\end{proof}
We remark that Proposition \ref{prop_min_max_length} cannot be improved to a length of ${|\rho| + 2 \, \mdd(\rho) - 1}$. Indeed, the pattern $013542$ is of length 6 and $\mdd$ 2, and all 34 of its minimal inversion sequences are of length 8. The process we used to generate these minimal inversion sequences is explained in Section \ref{sec_exhaustive}, and Table \ref{table_ISBT} (on page \pageref{table_ISBT}) shows that every pattern $\rho$ of length 5 or less admits some $\rho$-minimal inversion sequence of length at least ${|\rho| + 2 \, \mdd(\rho) - 1}$.

\section{Enumeration of minimal inversion sequences} \label{sec_enumeration}

\subsection{Patterns beginning by a large value} \label{sec_colours}

In this section, we enumerate $\rho$-minimal inversion sequences for any pattern $\rho$ such that $\rho_1 = \mdd(\rho)$, i.e. $\rho_1 > \rho_i-i$ for all $i$. As we shall see, this assumption gives rise to a very simple structure for $\rho$-minimal inversion sequences (in particular, the number of such sequences depends only on the $\mdd$ of $\rho$). Moreover, this assumption is often satisfied by patterns of small size. To be precise, a computation shows that patterns $\rho$ such that $\rho_1 = \mdd(\rho)$ account for more than half of $\P_k$ for any $k \in [1, 11]$, and more than 10\% of $\P_k$ up to $k = 231$.

For all $n \geq m \geq 0$, let $\mathcal A_{n,m}$ be the set of coloured inversion sequences of length $n$ which satisfy the following conditions:
\begin{itemize}
    \item every \emph{value} in $[0,n-1]$ is assigned a colour, either blue or red\footnote{In what follows, we additionally underline blue values and overline red values so that colour is not required to tell them apart.}, and all occurrences of this value in the inversion sequence receive its colour,
    \item there are $m$ distinct blue values and $n - m$ distinct red values,
    \item each red value occurs at least twice in the sequence.
\end{itemize}
Note that any value of $[0,n-1]$ which does not occur in such a sequence is implicitly blue. For instance, $\b 0 \b 0 \r 1 \r 3 \b 2 \r 3 \r 1 \b 6 \r 1 \in \mathcal A_{9,7}$ and the corresponding colouring of values is $\{\b 0, \r 1, \b 2, \r 3, \b 4, \b 5, \b 6, \b 7, \b 8\}$, but $\r 0 \b 1 \b 2 \notin \mathcal A_{3,2}$ since the red $\r 0$ appears only once.
Note also that $\mathcal A_{n,m}$ is empty when $n > 2m$, since its sequences must have at least $2(n-m)$ red entries. For any sequence $\alpha \in \mathcal A_{n,m}$, let $\uncolour(\alpha) \in \I_n$ be the uncoloured underlying inversion sequence.

\begin{prop} \label{prop_coloured_prefix}
    Let $k \geq 1$, $\,m \geq 0$, and let $\rho \in \P_k$ be such that $\rho_1 = \mdd(\rho) = m$.
    For all $n \geq 0$, the number of $\rho$-minimal inversion sequences of length $n+k$ is $|\mathcal A_{n,m}|$.
\end{prop} 
\begin{proof}
    We present a simple bijection between $\rho$-minimal inversion sequences of length $n+k$ and $\mathcal A_{n,m}$, depicted in Figure \ref{Fig_coloured_prefix}.

     Let $\sigma$ be a $\rho$-minimal inversion sequence of length $n+k$. By Proposition \ref{prop_sat_pattern}, $\sigma$ has a unique occurrence $\rho' := (\sigma_{i})_{i \in [n+1,n+k]}$ of $\rho$. Let $\alpha$ be the coloured inversion sequence of length $n$ defined by 
     $$\forall i \in [1,n], \;\alpha_i := \begin{cases}
         \b{\sigma_i} & \text{if } \; \sigma_i \in \{\rho'_j\}_{j \in [1,k]}, \\
         \r{\sigma_i} & \text{otherwise}.
     \end{cases}$$
    Note that every red value of $\alpha$ appears at least twice, by item 2 of Proposition \ref{prop_minimal_characterization}. Moreover, the number of distinct red values in $\alpha$ is $\dist(\sigma) - \dist(\rho)$, which equals $n-m$ by Lemma \ref{lem_len-dist-mdd}.
    It follows that $\alpha \in \mathcal A_{n,m}$.
    
    To describe the inverse bijection, let $\alpha \in \mathcal A_{n,m}$, and let $\rho'$ be the sequence obtained from $\rho$ by mapping the $i$-th smallest value in $[0,\max(\rho)]$ to the $i$-th smallest blue value of $\alpha$, with the convention that any value greater than or equal to $n$ is blue (since it cannot appear in $\alpha$). More formally, $\rho'$ can be defined as the only sequence which is order-isomorphic to $\rho$ (i.e. such that $\rho \preceq \rho' \preceq \rho)$ and whose set of values is 
    $$[0,\max(\rho)+n-m] \, \backslash \, \{v \in [0,n-1] \; : \; \r v \in \{\alpha_i\}_{i \in [1,n]}\}.$$
    Now we show that $\sigma := \uncolour({\alpha}) \cdot \rho'$ is a $\rho$-minimal inversion sequence.
    \begin{itemize}
        \item For all $i \in [1,k]$, we have $\rho'_i \leq \rho_i + n-m$, hence $\mdd(\rho') \leq \mdd(\rho) + (n-m) = n$, therefore $\sigma$ is an inversion sequence. 
        \item There are exactly $m$ values of $\rho$ smaller than $\rho_1$, and $\alpha$ has exactly $m$ blue values in $[0,n-1]$, so it follows from the construction of $\rho'$ that $\rho'_1 = n$. In other words, $\sigma$ has a saturated entry at position $n+1$.
        \item Every value in $\{\sigma_{i}\}_{i \in [1,n]}$ which does not appear in $\rho'$ occurs at least twice in $\sigma$, since it corresponds to a red value of $\alpha$.
        \item Lastly, we show that $\sigma$ contains only one occurrence of the pattern $\rho$. Let $\ell \in [1,n+k]$ be the position of the leftmost term of an occurrence of $\rho$ in $\sigma$. In particular, we have $\ell \geq m+1$ since $\sigma_\ell \geq m$, and $\ell \leq n+1$ since there must be at least $k-1$ entries of $\sigma$ to its right. If $\ell = n+1$, then clearly the only occurrence of $\rho$ whose first term is $\sigma_\ell$ is $(\sigma_i)_{i \in [n+1,n+k]} = \rho'$.
        
        Assume for the sake of contradiction that $\ell \leq n$. Let $v := \sigma_{\ell}$, and let
        \begin{itemize}
            \item $V := \{\sigma_i \; : \; i > \ell \; \text{ and } \; \sigma_i < v \}$ be the set of values of $\sigma$ which are less than $v$ and appear to the right of position $\ell$,
            \item $U := \{i \in [0,n-1] \; : \; \r i \in \{\alpha_j\}_{j \leq \ell} \; \text{ and } \; \r i \notin \{\alpha_j\}_{j > \ell}\}$ be the set of values that are red in $\alpha$ and do not appear to the right of position $\ell$,
            \item $U^- :=\{i \in U \; : \; i < v\}$, 
            \item $U^+ := \{i \in U \; : \; i \geq v\}$.
        \end{itemize}
        Since $\sigma_\ell$ takes the role of $\rho_1 = m$ in an occurrence of $\rho$, we have $|V| \geq m$. Notice that $V$ must contain all values of $[0,v-1]$ that are blue in $\alpha$, since these values all appear in $\rho'$, and therefore to the right of position $\ell$ in $\sigma$. It follows that $V = [0,v-1] \backslash U^-$. 
        The total number of red values in $\alpha$ is $n-m$, and at most $n-\ell$ of them can fit in the interval of positions $[\ell+1,n]$, so we have $|U| \geq \ell - m$. Moreover, each value of $U^+$ must appear at least twice in the interval of positions $[v+1,\ell]$, so $|U^+| \leq (\ell - v)/2 \leq \ell - v - 1$. In conclusion, we obtain the following inequality:
        $$|V| = v - |U^-| = v - |U| + |U^+| \leq v - (\ell - m) + (\ell - v - 1) = m -1,$$
        a contradiction.
    \end{itemize}
    It follows from Proposition \ref{prop_minimal_characterization} that $\sigma$ is a $\rho$-minimal inversion sequence.
\end{proof} 
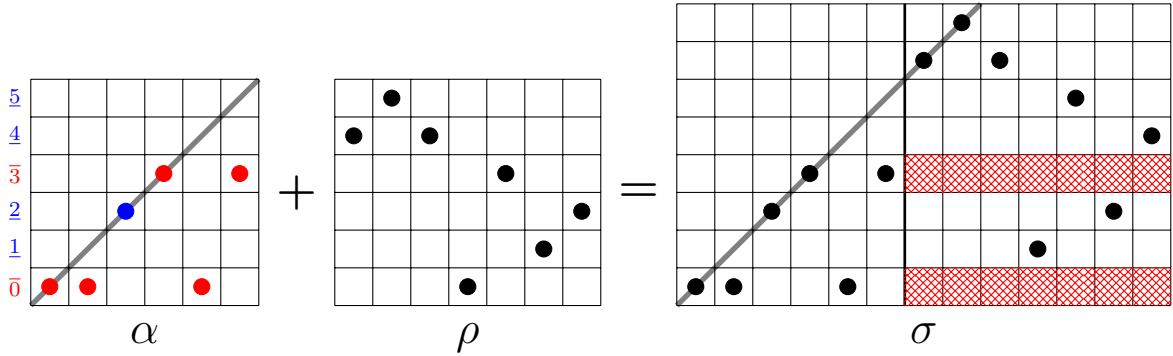
\begin{figure}[ht]
\centering
\begin{tikzpicture}
\begin{scope}[scale = 0.5]
    \draw[line width = 2pt, color = gray] (0,0)--(6,6);
    \draw (0,0) grid (6,6);
    \DrawPoint[red]{1}{0};
    \DrawPoint[red]{2}{0};
    \DrawPoint[blue]{3}{2};
    \DrawPoint[red]{4}{3};
    \DrawPoint[red]{5}{0};
    \DrawPoint[red]{6}{3};
    \draw (0,0.5) node[left, font=\scriptsize] {$\r 0$};
    \draw (0,1.5) node[left, font=\scriptsize] {$\b 1$};
    \draw (0,2.5) node[left, font=\scriptsize] {$\b 2$};
    \draw (0,3.5) node[left, font=\scriptsize] {$\r 3$};
    \draw (0,4.5) node[left, font=\scriptsize] {$\b 4$};
    \draw (0,5.5) node[left, font=\scriptsize] {$\b 5$};
    \draw (3,-0.2) node[below, font=\LARGE] {$\alpha$};

    \draw (7,3) node[font=\LARGE] {\textbf +};

    \begin{scope}[xshift = 8cm]
        \draw (0,0) grid (7,6);
        \DrawPoint{1}{4};
        \DrawPoint{2}{5};
        \DrawPoint{3}{4};
        \DrawPoint{4}{0};
        \DrawPoint{5}{3};
        \DrawPoint{6}{1};
        \DrawPoint{7}{2};
        \draw (3.5,-0.2) node[below, font=\LARGE] {$\rho$};
    \end{scope}
    \draw (16,3) node[font=\LARGE] {\textbf =};
    \begin{scope}[xshift = 17cm]
        \draw[line width = 2pt, color = gray] (0,0)--(8,8);
        \draw (0,0) grid (13,8);
        \DrawPoint{1}{0};
        \DrawPoint{2}{0};
        \DrawPoint{3}{2};
        \DrawPoint{4}{3};
        \DrawPoint{5}{0};
        \DrawPoint{6}{3};
        \draw[line width = 1pt] (6,0)--(6,8);
        \fill[pattern = crosshatch, pattern color = red] (6,0) rectangle (13,1);
        \fill[pattern = crosshatch, pattern color = red] (6,3) rectangle (13,4);

        \DrawPoint{7}{6};
        \DrawPoint{8}{7};
        \DrawPoint{9}{6};
        \DrawPoint{10}{1};
        \DrawPoint{11}{5};
        \DrawPoint{12}{2};
        \DrawPoint{13}{4};
        \draw (6.5,-0.2) node[below, font=\LARGE] {$\sigma$};
    \end{scope}
\end{scope}
\end{tikzpicture}
\caption{The coloured inversion sequence $\alpha = \r 0 \r 0 \b 2 \r 3 \r 0 \r 3$ and the pattern $\rho = 4540312$ together form the $\rho$-minimal inversion sequence ${\sigma = 0023036761524}$.}
\label{Fig_coloured_prefix}
\end{figure}

In order to count our coloured inversion sequences, we introduce another object.
For all $n \geq m \geq 1$, let $\mathcal B_{n,m}$ be the set of non-plane coloured labelled rooted trees on $n$ nodes satisfying the following:
\begin{enumerate}
    \item each node is labelled by a value in $[0,n-1]$ and coloured blue or red,
    \item each value in $[0,n-1]$ labels exactly one node,
    \item the tree is \emph{increasing}, i.e. each non-root node has a greater label than its parent,
    \item there are $m$ blue nodes and $n-m$ red nodes,
    \item each red node has at least two children.
\end{enumerate}

Let $\mathcal A := \bigcup_{n,m \geq 0} \mathcal A_{n,m}$ and $\mathcal B := \bigcup_{n,m \geq 1} \mathcal B_{n,m}$. We now establish a one-to-one correspondence between these two families.
Let $\phi \; : \; \mathcal A \to \mathcal B$ 
be the function which turns a coloured inversion sequence $\sigma$ into the tree $\tau$ defined as follows:
\begin{itemize}
    \item the root of $\tau$ has label $0$,
    \item the labels of the children of the node labelled $v$ in $\tau$ are the positions of the value $v$ in $\sigma$,
    \item the node labelled $v$ in $\tau$ has the same colour as the value $v$ in $\sigma$; the node labelled $|\sigma|$ is blue by convention.
\end{itemize}
Figure \ref{Fig_inv_tree_bijection} depicts the function $\phi$ on an example.

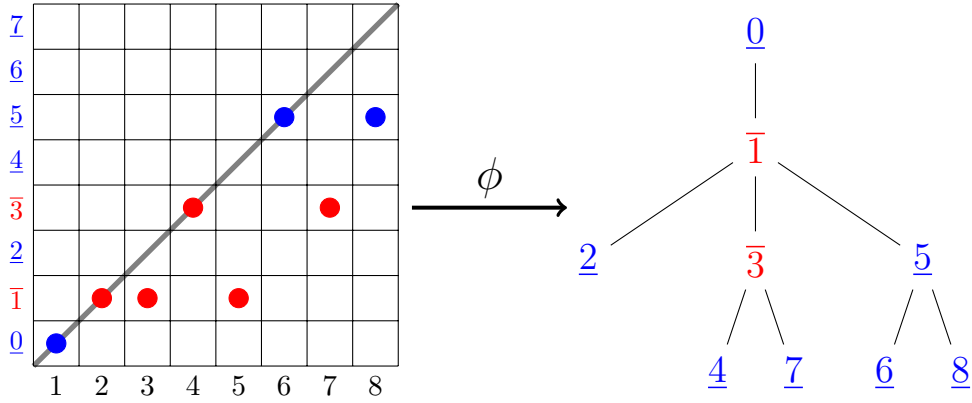
\begin{figure}[ht]
\centering
\begin{tikzpicture}
\begin{scope}[scale = 0.6]
    \draw[line width = 2pt, gray] (0,0) -- (8,8);
    \draw (0,0) grid (8,8);
    \DrawPoint[blue]{1}{0};
    \DrawPoint[red]{2}{1};
    \DrawPoint[red]{3}{1};
    \DrawPoint[red]{4}{3};
    \DrawPoint[red]{5}{1};
    \DrawPoint[blue]{6}{5};
    \DrawPoint[red]{7}{3};
    \DrawPoint[blue]{8}{5};
    \draw (0,0.5) node[left, font=\small] {$\b 0$};
    \draw (0,1.5) node[left, font=\small] {$\r 1$};
    \draw (0,2.5) node[left, font=\small] {$\b 2$};
    \draw (0,3.5) node[left, font=\small] {$\r 3$};
    \draw (0,4.5) node[left, font=\small] {$\b 4$};
    \draw (0,5.5) node[left, font=\small] {$\b 5$};
    \draw (0,6.5) node[left, font=\small] {$\b 6$};
    \draw (0,7.5) node[left, font=\small] {$\b 7$};
    \draw (0.5,0) node[below] {1};
    \draw (1.5,0) node[below] {2};
    \draw (2.5,0) node[below] {3};
    \draw (3.5,0) node[below] {4};
    \draw (4.5,0) node[below] {5};
    \draw (5.5,0) node[below] {6};
    \draw (6.5,0) node[below] {7};
    \draw (7.5,0) node[below] {8};
    
    \draw [->, line width = 1.5pt] (8.3, 3.5) -- (11.7, 3.5);
    \draw (10,3.5) node[above, font=\LARGE] {$\phi$};
\end{scope}

\begin{scope}[xshift = 95mm, yshift = 44mm, node font = \Large]
\tikzstyle{level 1}=[level distance=15mm]
\tikzstyle{level 2}=[sibling distance=22mm]
\tikzstyle{level 3}=[sibling distance=10mm]
\node{$\b 0$}
child {node{$\r 1$}
    child {node{$\b 2$}}
    child {node{$\r 3$}
        child {node{$\b 4$}}
        child {node{$\b 7$}}
    }
    child {node{$\b 5$}
        child {node{$\b 6$}}
        child {node{$\b 8$}}
    }
};
\end{scope}
\end{tikzpicture}
\caption{Illustration of the function $\phi$.}
\label{Fig_inv_tree_bijection}
\end{figure}

It is easy to check that for any $\sigma \in \mathcal A_{n,m}$, its image $\tau := \phi(\sigma)$ satisfies all 5 items in the definition of $\mathcal B_{n+1,m+1}$:
\begin{itemize}
    \item item 1 is satisfied since $\phi$ assigns a label and a colour to every node of $\tau$,
    \item item 2 is satisfied since the root of $\tau$ has label 0, and every integer in $[1,n]$ is the position of some value of $\sigma$,
    \item item 3 is satisfied since $\sigma$ is an inversion sequence: a value $v$ can only appear at position at least $v+1$ in $\sigma$,
    \item item 4 is satisfied since $\sigma$ has $m$ blue values and $n-m$ red values in $[0,n-1]$, and the node labelled $n$ is always blue,
    \item item 5 is satisfied since red values of $\sigma$ appear at least twice.
\end{itemize}
It is also easy to reverse $\phi$ and construct a coloured inversion sequence $\sigma \in \mathcal A$ from a tree $\tau \in \mathcal B$, since $\tau$ clearly determines the colour and positions of each value of $\sigma$. It follows that $\phi$ is a bijection between $\mathcal A_{n,m}$ and $\mathcal B_{n+1,m+1}$ for all $n,m \geq 0$.

For all $n < m$, let $\mathcal A_{n,m} := \emptyset$ and $\mathcal B_{n,m} := \emptyset$. Let $B(x,y) := \sum_{n,m \geq 1} \frac{|\mathcal B_{n,m}|}{n!} x^n y^m$ be the mixed generating function of $\mathcal B$ that is exponential in $x$ and ordinary in $y$.
\begin{thm}
    The formal power series $B(x,y)$ satisfies the equations
    $$\frac{\partial B}{\partial x} (x,y) = (y+1)e^{B(x,y)} - B(x,y) - 1, \quad B(0,y)=0.$$
\end{thm}
\begin{proof}
We use the symbolic method from \cite[Part A]{Flajolet_Sedgewick}. The class $\mathcal B$ is described by the recursive specification below, where $\mathcal X$ is an atomic class representing a single labelled node, $\mu$ is a neutral object marking the colour blue, and the symbol $^\square$ indicates the \emph{boxed product} \cite[page 139]{Flajolet_Sedgewick}. 
$$\mathcal B = \mu \, \mathcal X^\square \star \textsc{Set}(\mathcal B) + \mathcal X^\square \star \textsc{Set}_{\geq 2}(\mathcal B)$$
This specification translates to the following functional equation:
\begin{equation} \label{eq_B}
    B(x,y) = \int_0^x \left ( y \, e^{B(t,y)} +  e^{B(t,y)} - B(t,y) - 1 \right ) dt.
\end{equation}
The proof is easily concluded by taking the derivative of \eqref{eq_B} with respect to $x$, and by evaluating \eqref{eq_B} at $x=0$.
\end{proof}

Let $k \geq 1$, $\,m \geq 0$, and let $\rho \in \P_k$ be such that $\rho_1 = \mdd(\rho) = m$. By Proposition \ref{prop_coloured_prefix}, the number of $\rho$-minimal inversion sequences of length $n+k$ is $|\mathcal A_{n,m}|$, or equivalently $|\mathcal B_{n+1,m+1}| = (n+1)! [x^{n+1} y^{m+1}] B(x,y)$. These numbers are shown for small values of $n$ and $m$ in Table \ref{table_coloured_inv}, and we added them to \cite{OEIS} as entry A393523. The total number of $\rho$-minimal inversion sequences is $\sum_{n = m}^{2m} |\mathcal A_{n,m}|$, now entry A393524 of \cite{OEIS}.

\begin{table}
$$\begin{array}{|c|c|c|c|c|c|c|c|c|}
\hline
|\mathcal A_{n,m}| & m=0 & m=1 & m=2 & m=3 & m=4 & m=5 & m = 6 & m = 7 \\
\hline
n=0 & 1 & 0 & 0 & 0 & 0 & 0 & 0 & 0 \\
\hline
n=1 & 0 & 1 & 0 & 0 & 0 & 0 & 0 & 0\\
\hline
n=2 & 0 & 1 & 2 & 0 & 0 & 0 & 0 & 0 \\
\hline
n=3 & 0 & 0 & 5 & 6 & 0 & 0 & 0 & 0 \\
\hline
n=4 & 0 & 0 & 4 & 27 & 24 & 0 & 0 & 0 \\
\hline
n=5 & 0 & 0 & 0 & 49 & 168 & 120 & 0 & 0 \\
\hline
n=6 & 0 & 0 & 0 & 34 & 515 & 1200 & 720 & 0 \\
\hline
n=7 & 0 & 0 & 0 & 0 & 790 & 5471 & 9720 & 5040 \\
\hline
\end{array}$$
\caption{The numbers $|\mathcal A_{n,m}|$ for $n,m \in [0,7]$.}
\label{table_coloured_inv}
\end{table}

We end this subsection with some remarks on the minimal-length and maximal-length cases.
\begin{itemize}
    \item[--] The $\rho$-minimal inversion sequences of minimal length are in bijection with $\mathcal A_{m,m}$. Inversion sequences in $\mathcal A_{m,m}$ only have blue values, and they need not satisfy any additional constraints, so they are trivially in bijection with $\I_m$. Similarly, $\mathcal B_{m+1,m+1}$ is trivially in bijection with the set of (non-plane) increasing trees on $m+1$ nodes.
    Restricting the function $\phi$ to this particular case yields a simple bijection between inversion sequences and increasing trees. This correspondence appears in \cite[page 145]{Flajolet_Sedgewick}, it is used to enumerate inversion sequences avoiding the pattern 000 in \cite{Corteel_Martinez_Savage_Weselcouch_2016}, and generalized for a constant pattern of any length in \cite{Hong_Li_2022}.
    \item[--] The $\rho$-minimal inversion sequences of maximal length are in bijection with $\mathcal A_{2m,m}$. Inversion sequences in $\mathcal A_{2m,m}$ have $m$ red values and $m$ blue values, each red value appears exactly twice, and blue values do not appear at all.
    Similarly, $\mathcal B_{2m+1,m+1}$ is the set of (non-plane) increasing trees having $m$ red nodes and $m+1$ blue nodes, such that each red node has exactly two children, and all blue nodes are leaves. This family of trees was studied by Poupard \cite{Poupard_1989}, who found that their exponential generating function (counting trees with respect to their size $2m+1$) is $\sqrt 2 \tan(x/\sqrt 2)$. Their counting sequence appears in \cite{OEIS} as entry A002105.
\end{itemize}

\subsection{Exhaustive generation for small patterns} \label{sec_exhaustive}
In this subsection, we use computer programs to generate every $\rho$-minimal inversion sequence for a pattern $\rho$ of small length. We know from Corollary \ref{corol_length} that the length of every $\rho$-minimal inversion sequence is between $|\rho| + \mdd(\rho)$ and $|\rho| + 2 \,\mdd(\rho)$, so an easy approach would be to generate each sequence of $\I \cap \P$ whose length lies in this interval, and then test whether it is $\rho$-minimal using either the definition of $\rho$-minimal sequences (checking whether there is a smaller element in the poset $(\I\P[\rho], \preceq)$), or the characterization of Proposition \ref{prop_minimal_characterization} (checking whether every occurrence of $\rho$ satisfies some conditions). While this approach works for very small patterns, its limits are already noticeable for some patterns of length 5: their minimal inversion sequences can reach a length of 13, and there are 363\,674\,407 sequences in $\I_{13} \cap \P_{13}$ to test (see Section \ref{sec_I_inter_P} for the enumeration of $\I \cap \P$). 

Instead, we make use of some properties of minimal inversion sequences from Section \ref{sec_structure} to generate them more efficiently.
Assume that we are given a nonempty pattern $\rho$ with $m := \mdd(\rho)$ and $s := \min(\Sat(\rho))$, and that we wish to generate every $\rho$-minimal inversion sequence.
The general idea of our approach is to start from the pattern $\rho$, and build an inversion sequence $\sigma$ around it by inserting some new entries, possibly shifting some values of $\rho$ in the process.
This way, we are guaranteed to construct an inversion sequence which contains at least one occurrence of $\rho$ (we later refer to it as the \emph{starting occurrence} of $\rho$).
Additionally, Section \ref{sec_structure} gives several conditions that the newly inserted entries must satisfy for $\sigma$ to be $\rho$-minimal. In practice, we generate all such sequences $\sigma$ which belong to $\I \cap \P$ and satisfy the following conditions.
\begin{itemize}
    \item By items 1 and 2 of Proposition \ref{prop_sat_pattern}, each new entry inserted must be to the left and below the leftmost saturated entry of the starting occurrence of $\rho$ (that is $\rho_s$ at the beginning of our construction).
    \item By Lemma \ref{lem_len-dist-mdd}, if $n := |\sigma| - |\rho|$ is the number of new entries inserted in $\rho$, and $v := \dist(\sigma) - \dist(\rho)$ is the number of new values introduced by these entries, then $n-v = m$.
    \item By Corollary \ref{corol_length}, the number $n$ of new entries belongs to the interval $[m,2m]$.
    \item By item 2 of Proposition \ref{prop_minimal_characterization}, each of the $v$ new values is inserted at least twice.
\end{itemize}
If $\rho_1 = m$, then we obtain precisely the construction of Proposition \ref{prop_coloured_prefix} (the colour red marks new values, and the colour blue marks values which also appear in the starting occurrence of $\rho$).
This case is easy because, as we showed, inserting new entries does not create any other occurrences of $\rho$. In the general case, things are not so simple: while our construction guarantees that the starting occurrence of $\rho$ satisfies the conditions of Proposition \ref{prop_minimal_characterization}, we cannot guarantee this property for \emph{every} occurrence of $\rho$ in $\sigma$.
Consequently, we must test the $\rho$-minimality of each inversion sequence generated by this process. Nevertheless, this is much faster than the naive approach presented earlier, since the set of sequences to test is now significantly smaller. With this approach, any modern computer can generate all $\rho$-minimal inversion sequences for any pattern $\rho$ of length up to 7, taking at most a few minutes. This method also works well for patterns of greater length that have a small $\mdd$.

In order to better present some of this data, we introduce a new notation. Given a nonempty pattern $\rho$, we denote by $\ISBT(\rho)$ the \emph{inversion sequence basis type} of $\rho$, that is, the counting sequence of $\rho$-minimal inversion sequences of length ranging from $|\rho| + \mdd(\rho)$ to $|\rho| + 2 \, \mdd(\rho)$. For instance, if $\rho = 0312$, we have $|\rho| = 4$ and $\mdd(\rho) = 2$, so the length of every $\rho$-minimal inversion sequence is between 4 and 6; there are 6 of length 4, 5 of length 5, and none of length 6, so $\ISBT(0312) = (6,5,0)$. Note that $\ISBT(\rho) = (1)$ if and only if $\rho$ is an inversion sequence.

In Table \ref{table_ISBT}, we catalog the inversion sequence basis type of every pattern $\rho$ of length 1 through 5, using a computer program which generates every $\rho$-minimal inversion sequence. It is remarkable that there are only 27 distinct basis types for those 633 patterns.

\begin{table}
\centering
\begin{tabular}{|c|c|}
    \hline
    $\ISBT$ & Patterns \\
    \hline
    $(1)$ & Any $\rho$ such that $\rho_1 = \mdd(\rho) = 0$. \\
    \hline
    $(1, 1)$ & Any $\rho$ such that $\rho_1 = \mdd(\rho) = 1$. \\
    \hline
    $(2, 0)$ & \begin{minipage}{300pt} \strut \centering
    021, 0211, 0212, 0213, 0221, 0231, 02111, 02112, 02113, 02121, 02122, 02123, 02131, 02132, 02133, 02134, 02143, 02211, 02212, 02213, 02221, 02231, 02311, 02312, 02313, 02314, 02321, 02331, 02341
    \strut \end{minipage} \\
    \hline
    $(2, 3)$ & \begin{minipage}{300pt} \strut \centering
    0201, 0210, 02001, 02010, 02011, 02012, 02013, 02021, 02031, 02100, 02101, 02102, 02103, 02110, 02120, 02130, 02201, 02210, 02301, 02310
    \strut \end{minipage} \\
    \hline
    $(4, 0)$ & \begin{minipage}{300pt} \strut \centering
    0132, 01322, 01323, 01324, 01332, 01342
    \strut \end{minipage} \\
    \hline
    $(4, 1)$ & \begin{minipage}{300pt} \strut \centering
    01312, 01321
    \strut \end{minipage} \\
    \hline
    $(4, 3)$ & \begin{minipage}{300pt} \strut \centering
    01302, 01320
    \strut \end{minipage} \\
    \hline
    $(4, 4)$ & \begin{minipage}{300pt} \strut \centering
    00312, 00321
    \strut \end{minipage} \\
    \hline
    $(7, 0)$ & \begin{minipage}{300pt} \strut \centering
    01243
    \strut \end{minipage} \\
    \hline
    $(2, 5, 4)$ & Any $\rho$ such that $\rho_1 = \mdd(\rho) = 2$. \\
    \hline
    $(4, 8, 5)$ & \begin{minipage}{300pt} \strut \centering
    1302, 13024, 13042
    \strut \end{minipage} \\
    \hline
    $(4, 8, 6)$ & \begin{minipage}{300pt} \strut \centering
    1320, 13002, 13020, 13022, 13023, 13032, 13200, 13202, 13203, 13204, 13220, 13230, 13240, 13302, 13320, 13402, 13420
    \strut \end{minipage} \\
    \hline
    $(4, 11, 9)$ & \begin{minipage}{300pt} \strut \centering
    13012, 13021, 13102, 13120, 13201, 13210
    \strut \end{minipage} \\
    \hline
    $(6, 5, 0)$ & \begin{minipage}{300pt} \strut \centering
    0312, 0321, 03112, 03121, 03122, 03123, 03124, 03132, 03142, 03211, 03212, 03213, 03214, 03221, 03231, 03241, 03312, 03321, 03412, 03421
    \strut \end{minipage} \\
    \hline
    $(6, 21, 21)$ & \begin{minipage}{300pt} \strut \centering
    03012, 03021, 03102, 03120, 03201, 03210
    \strut \end{minipage} \\
    \hline
    $(7, 10, 3)$ & \begin{minipage}{300pt} \strut \centering
    10423, 10432
    \strut \end{minipage} \\
    \hline
    $(7, 13, 9)$ & \begin{minipage}{300pt} \strut \centering
    12403, 12430
    \strut \end{minipage} \\
    \hline
    $(12, 6, 0)$ & \begin{minipage}{300pt} \strut \centering
    02413, 02431
    \strut \end{minipage} \\
    \hline
    $(18, 4, 0)$ & \begin{minipage}{300pt} \strut \centering
    01423, 01432
    \strut \end{minipage} \\
    \hline
    $(6, 27, 49, 34)$ & Any $\rho$ such that $\rho_1 = \mdd(\rho) = 3$. \\
    \hline
    $(12, 50, 86, 57)$ & \begin{minipage}{300pt} \strut \centering
    24013
    \strut \end{minipage} \\
    \hline
    $(12, 50, 87, 59)$ & \begin{minipage}{300pt} \strut \centering
    24130
    \strut \end{minipage} \\
    \hline
    $(12, 50, 88, 60)$ & \begin{minipage}{300pt} \strut \centering
    24031, 24103, 24301, 24310
    \strut \end{minipage} \\
    \hline
    $(18, 61, 95, 58)$ & \begin{minipage}{300pt} \strut \centering
    14023, 14203
    \strut \end{minipage} \\
    \hline
    $(18, 61, 95, 60)$ & \begin{minipage}{300pt} \strut \centering
    14032, 14230, 14302, 14320
    \strut \end{minipage} \\
    \hline
    $(24, 54, 38, 0)$ & \begin{minipage}{300pt} \strut \centering
    04123, 04132, 04213, 04231, 04312, 04321
    \strut \end{minipage} \\
    \hline
    $(24, 168, 515, 790, 496)$ & Any $\rho$ such that $\rho_1 = \mdd(\rho) = 4$. \\
    \hline
\end{tabular}
\caption{The inversion sequence basis type of every pattern of length at most 5.}
\label{table_ISBT}
\end{table}

\section{Enumeration of \texorpdfstring{$\I \cap \P$}{I ∩ P}} \label{sec_I_inter_P}

Minimal inversion sequences for a pattern are, by definition, objects that belong to the combinatorial class $\I \cap \P$. Surprisingly, it appears this class has not yet been enumerated, and we now make a short detour to address this problem.

For all $n \geq 1$ and $k \geq 0$, let $\I\P_{n,k} := \{\sigma \in \I \cap \P \; : \; |\sigma| = n \; \text{ and } \; \max(\sigma) = k\}$. Equivalently, we have $\I\P_{n,k} = \{\sigma \in \I_n \; : \; \{\sigma_i\}_{i \in [1,n]} = [0,k]\}$. For all $n,k \geq 1$, let $\T_{n,k}$ be the set of non-plane increasing trees having $n$ internal nodes labelled by the integers $[0,n-1]$, and $k$ leaves labelled by the integers $[n, n+k-1]$.

As we remarked at the end of Section \ref{sec_colours}, the function $\phi$ of Section \ref{sec_colours} can be restricted to a bijection between $\I_n$ and (non-plane) increasing trees on $n+1$ nodes. Further restricting its domain to $\I \cap \P$ yields another nice correspondence: $\phi$ is also a bijection between $\I\P_{n,k}$ and $\T_{k+1,n-k}$. This is easy to see since $\I\P_{n,k}$ is the subset of inversion sequences of length $n$ whose values are $[0,k]$, and $\T_{k+1,n-k}$ is the subset of increasing trees on $n+1$ nodes whose internal node labels are $[0,k]$.

We now count those trees. Let $\T := \coprod_{n,k \geq 1} \T_{n,k}$, and let $T(x,y) := \sum_{n,k \geq 1} \frac{|\T_{n,k}| x^n y^k}{n! \, k!}$ be their doubly exponential generating function.

\begin{thm} \label{thm_poly_Bernoulli}
    $$T(x,y) = \log \left (\frac{1}{e^x + e^y - e^{x+y}} \right )$$
\end{thm}
\begin{proof} 
We once again use the symbolic method from \cite[Part A]{Flajolet_Sedgewick}. In order to provide a specification for $\mathcal T$, we consider that objects of a set $\mathcal T_{n,k}$ have two different labellings: one for internal nodes, corresponding to the labels $[0,n-1]$, and one for leaves, corresponding to the labels $[n,n+k-1]$. While the book \cite{Flajolet_Sedgewick} does not explicitly treat the case of objects with two labellings and doubly exponential generating functions, this is a harmless generalization of their approach.
The class $\T$ is described by the recursive specification below, where
\begin{itemize}
    \item $\mathcal X$ is an atomic class consisting of a single labelled object (an internal node),
    \item $\mathcal Y$ is an atomic class consisting of a single labelled object (a leaf),
    \item the boxed product is taken with respect to the labelling of internal nodes only.
\end{itemize}
$$\T = \mathcal X^{\square} \star \textsc{Set}_{\geq 1}(\T +\mathcal Y)$$
This definition corresponds to the following equation for $T(x,y)$:
$$T(x,y) = \int_0^x (e^{T(t,y) + y} - 1) dt.$$
Equivalently, the formal power series $T(x,y)$ is characterized by the two equations below:
$$\frac{\partial T}{\partial x}(x,y) = e^{T(x,y) + y} -1, \quad T(0,y) = 0.$$
It is easy to verify that the function $\log \left (1/(e^x + e^y - e^{x+y}) \right )$ satisfies both equations, concluding the proof.
\end{proof}

\begin{table}[ht]
$$\begin{array}{|c|c|c|c|c|c|c|c|}
\hline
|\T_{n,k}| & k=1 & k=2 & k=3 & k=4 & k=5 & k=6 & k=7 \\
\hline
n=1 & 1 & 1 & 1 & 1 & 1 & 1 & 1\\
\hline
n=2 & 1 & 3 & 7 & 15 & 31 & 63 & 127 \\
\hline
n=3 & 1 & 7 & 31 & 115 & 391 & 1267 & 3991 \\
\hline
n=4 & 1 & 15 & 115 & 675 & 3451 & 16275 & 72955 \\
\hline
n=5 & 1 & 31 & 391 & 3451 & 25231 & 164731 & 999391  \\
\hline
n=6 & 1 & 63 & 1267 & 16275 & 164731 & 1441923 & 11467387 \\
\hline
n=7 & 1 & 127 & 3991 & 72955 & 999391 & 11467387 & 116914351 \\
\hline
\end{array}$$
\caption{The numbers $|\T_{n,k}|$ for $n,k \in [1,7]$.}
\label{table_poly-Bernoulli}
\end{table}

We remark that a consequence of Theorem \ref{thm_poly_Bernoulli} is the symmetry $|\T_{n,k}| = |\T_{k,n}|$, which is not obvious from the definition of our increasing trees or inversion sequences.
Table \ref{table_poly-Bernoulli} shows the numbers $|\T_{n,k}|$ for small values of $n$ and $k$. These numbers appear in \cite{OEIS} as entry A136126. They are called the \emph{poly-Bernoulli numbers\footnote{The name ``poly-Bernoulli numbers" was originally given by Kaneko \cite{Kaneko_1997} to the numbers $|\T_{n+1,k}| + |\T_{n,k+1}|$.} of type C},
and have several known combinatorial interpretations, see e.g. \cite{Bényi_Hajnal_2017} or \cite{Knuth_2024}.

The numbers $|\I_n \cap \P_n|$ are the antidiagonal sums $\sum_{k = 0}^{n-1} |\T_{k+1,n-k}|$. For $n \in [1,7]$, these numbers are $1, 2, 5, 16, 63, 294, 1585$. They appear in \cite{OEIS} as entry A136127, for which no simple formula is known.

\section{Open questions} \label{sec_conclusion}

Our first question concerns the poly-Bernoulli numbers discussed above.
\begin{question}
    Find bijections between the families $\I\P_{n,k}$ or $\T_{n,k}$ of Section \ref{sec_I_inter_P} and other families of combinatorial objects known to be counted by the poly-Bernoulli numbers of type C.
\end{question}

One of the main questions that our work gives rise to is to enumerate, or at least better understand the set of minimal inversion sequences for any pattern. While we solved that question for patterns $\rho$ such that $\rho_1 = \mdd(\rho)$ in Section \ref{sec_colours}, the general case is much more difficult. Indeed, it is not even easy to determine the exact length of the longest $\rho$-minimal inversion sequence for any pattern $\rho$, although Corollary \ref{corol_length} and Proposition \ref{prop_min_max_length} tell us that it always lies between $|\rho| + 2 \, \mdd(\rho) -2$ and $|\rho| + 2 \, \mdd(\rho)$. We have done some preliminary study of this question, and we expect that a complete answer would be difficult to formulate, since it seems to require many different constructions of minimal inversion sequences depending on the pattern considered.
\begin{question}
    Find a way to determine the exact length of the longest $\rho$-minimal inversion sequence for any pattern $\rho$.
\end{question}

One area in which we expect progress to be relatively easy is to identify more classes of patterns that have a given $\ISBT$. In particular, we believe that the approach of Section \ref{sec_colours} could be generalized to patterns whose prefix to the left of the first saturated entry has a simple form. For instance, we expect this to work well with patterns of the form $0^z \cdot \rho$ where $z \geq 0$ and $\rho$ satisfies $\rho_1 = \mdd(\rho)$; a look at Table \ref{table_ISBT} suggests this should give a way of identifying patterns whose $\ISBT$ is $(2,0)$, $(2,3)$, $(6,5,0)$, $(6,21,21)$, or $(24, 54, 38, 0)$, among others.

Two patterns $\rho$ and $\rho'$ are said to be \emph{Wilf-equivalent} for inversion sequences if there is the same number of $\rho$-avoiding as $\rho'$-avoiding sequences in $\I_n$ for all $n \geq 0$. The complete classification of Wilf-equivalent patterns of length 3 in inversion sequences was done in \cite{Corteel_Martinez_Savage_Weselcouch_2016}, and an almost complete classification for patterns of length 4 appears in \cite{Hong_Li_2022}. From this data, we noticed that Wilf-equivalent patterns of length at most $4$ always have the same $\ISBT$.
\begin{question}
    Does Wilf-equivalence between two patterns always imply they have the same inversion sequence basis type?
\end{question}

A natural extension of our work is to study other \emph{minimal} combinatorial objects defined in a similar way for some notion of pattern containment. Some obvious choices of objects to consider are well-known subsets of inversion sequences in which patterns have also been studied, for the same pattern containment relation. We briefly discuss two such cases below.
\begin{itemize}
    \item \textbf{Restricted growth functions.} A \emph{record} (or strict left-to-right maximum) of an integer sequence $\sigma$ is an entry $\sigma_i$ such that $\sigma_i > \sigma_j$ for all $j < i$. 
    Let $$\REC(\sigma) := \{\sigma_i \; : \; \forall j<i, \; \sigma_i > \sigma_j\}$$ be the set of record values of $\sigma$. A \emph{restricted growth function} (see e.g. \cite{Bean_Bell_Ollson_2025}), or RGF for short, can be defined as an inversion sequence $\sigma$ such that $\REC(\sigma) = [0, \max(\sigma)]$. In particular, RGFs are always Cayley permutations.

    Given a pattern $\rho$, it is much easier to determine all $\rho$-minimal RGFs than $\rho$-minimal inversion sequences.
    Let $m := |[0,\max(\rho)] \backslash \REC(\rho)|$ be the number of values of $\rho$ which do not occur as a record. The value $m$ has a role analogous to $\mdd(\rho)$ in the inversion sequence case, since it serves as a measure of how much $\rho$ differs from a RGF. We claim that the $\rho$-minimal RGFs are precisely the sequences obtained by inserting one occurrence of each value of $[0,\max(\rho)] \backslash \REC(\rho)$ in $\rho$, in positions such that the resulting sequence is a RGF (each value must be inserted far enough to the left to be a record, and far enough to the right to avoid overtaking other records).    Indeed,
    \begin{itemize}
        \item it is clearly necessary to insert at least these entries in order to create a RGF which contains $\rho$,
        \item inserting any extra entries (and possibly shifting values in the process) is unnecessary and would create a RGF that is not $\rho$-minimal, since removing those extra entries (and taking the reduction of the resulting sequence) still yields a RGF which contains $\rho$.
    \end{itemize}
    In particular, $\rho$-minimal RGFs always have length $|\rho| + m$, they contain a single occurrence of $\rho$, and they have a much simpler structure than $\rho$-minimal inversion sequences.
    \item \textbf{Ascent sequences.} An \emph{ascent} (or consecutive pattern \underline{01}) of an integer sequence $\sigma$ is a pair of entries $(\sigma_i, \sigma_{i+1})$ such that $\sigma_i < \sigma_{i+1}$.
    Let $$\asc(\sigma) := |\{i \in [1,|\sigma|-1] \; : \; \sigma_i < \sigma_{i+1}\}|$$ be the number of ascents of $\sigma$. An \emph{ascent sequence} \cite{ascent_sequences_2010} is a nonnegative integer sequence $\sigma$ such that $\sigma_i \leq \asc((\sigma_j)_{j < i})$ for all $i$. In particular, ascent sequences are a subset of inversion sequences, and a superset of restricted growth functions.

    Given a pattern $\rho$, the set of $\rho$-minimal ascent sequences appears to be more difficult to study than its inversion sequence counterpart. We now make a few remarks about some of their differences.
    \begin{itemize}
        \item Some patterns that are not ascent sequences only have a single minimal ascent sequence. Consequently, some patterns having different lengths are Wilf-equivalent for ascent sequences. For instance, 010 is the only 10-minimal ascent sequence.
        \item Some patterns have ``gaps" in their ascent sequence basis type. For instance, there exist some 210-minimal ascent sequences of length 11 and 13, but none of length 12.
        \item We believe that a result on the length of $\rho$-minimal ascent sequences analogous to Corollary \ref{corol_length} could be given, albeit with an upper bound that is \emph{quadratic} in the length of $\rho$. More precisely, if $\rho$ is the decreasing pattern of length $k \geq 3$, we can construct a $\rho$-minimal ascent sequence of length $k^2 + 2k - 2$, and we conjecture this is the greatest achievable length.
        Such a sequence can be formed by concatenating $k-1$ factors\footnote{To be precise, the factor containing the values $[a,b]$ is obtained by interleaving the two nonincreasing sequences $(b-1, b-2, \dots, a+1) \cdot (a, a)$ and $(b,b) \cdot (b-1, b-2, \dots, a+1)$.} containing two occurrences of each value in the $k-1$ sets $[0,1],[2,4], [5,8], \dots, [(k^2-k-2)/2, (k^2+k-4)/2]$ (of respective sizes $2,3 \dots, k$), then appending an occurrence of $\rho$ using the values $[(k^2+k-2)/2, (k^2+3k-4)/2]$.
        For instance, with the notation $\texttt A := 10$, $\texttt B := 11$, and $\texttt C := 12$,
        $$0101 \cdot 342423 \cdot 78685756 \cdot \texttt{CBA}9$$ is a $3210$-minimal ascent sequence.
    \end{itemize}
    We hope to study $\rho$-minimal ascent sequences more thoroughly in the future.
\end{itemize}

\section*{Acknowledgements}
We thank Mathilde Bouvel for many insightful comments throughout the preparation of this work, and Emmanuel Jeandel for some remarks which improved its clarity.

\end{document}